\documentclass[12pt,oneside]{amsart}
\usepackage{ifpdf}
\ifpdf
  \usepackage[pdftex,colorlinks=true,%
              pdftitle={Rational six-cycles},%
              pdfauthor={Michael Stoll}]{hyperref}
\fi

\usepackage{amssymb}  
\usepackage{latexsym} 

\DeclareFontEncoding{OT2}{}{} 

\newcommand{\Z}{{\mathbb Z}}
\newcommand{\Q}{{\mathbb Q}}
\newcommand{\F}{{\mathbb F}}

\newcommand{\BP}{{\mathbb P}}
\newcommand{\To}{\longrightarrow}
\newcommand{\Pic}{\operatorname{Pic}}
\newcommand{\rank}{\operatorname{rank}}
\newcommand{\dyn}{{\text{\rm dyn}}}

\newtheorem{Theorem}{Theorem}
\newtheorem{Lemma}[Theorem]{Lemma}

\begin{document}

\title[Rational 6-cycles]{Rational 6-cycles under iteration \\ of quadratic polynomials}

\author{Michael Stoll}
\address{Department of Mathematics \\
         University of Bayreuth \\
	 95440 Bayreuth, Germany}
\email{Michael.Stoll@uni-bayreuth.de}
\date{October 8, 2008}


\begin{abstract}
  We present a proof, which is conditional on the Birch and Swinner\-ton-Dyer
  Conjecture for a specific abelian variety, that there do not exist
  rational numbers $x$ and~$c$ such that $x$ has exact period $N = 6$ under
  the iteration $x \mapsto x^2 + c$. This extends earlier results
  by Morton for $N = 4$ and by Flynn, Poonen and Schaefer for $N = 5$.
\end{abstract}

\maketitle


\section{Introduction}

In this note, we present a conditional proof that
there do not exist rational numbers $x$ and~$c$ such that the 
sequence defined by $x_0 = x$, $x_{n+1} = x_n^2 + c$ (for $n \ge 0$)
has exact period~6. The assumptions we have to make are that
the $L$-series of a certain genus~4 curve $X^{\dyn}_0(6)$ extends to
an entire function and satisfies the usual kind of functional equation,
and that the Jacobian of~$X^{\dyn}_0(6)$ satisfies the weak form of the
Birch and Swinnerton-Dyer conjecture.

This extends a series of investigations on rational cycles under
quadratic iteration. It is easy to see that fixed points and
2-cycles are each parametrized by a rational curve; the same is true
for 3-cycles. Morton~\cite{Morton} has shown that 4-cycles are
parametrized by the modular curve~$X_1(16)$; he used this to show
that there do not exist rational 4-cycles. Flynn, Poonen, and 
Schaefer~\cite{FPS} proved that there are no rational 5-cycles.
The present paper gives a conditional proof that there are no rational 6-cycles.
It is conjectured (see~\cite{Silverman}) that there is a universal
bound on the number of rational preperiodic points under quadratic iteration.
In view of the results obtained so far, it seems reasonable to expect
that there are no rational $N$-cycles when $N > 3$. Poonen~\cite{Poonen}
shows that this would imply that there can be at most 9 rational preperiodic
points.

Here is an overview of the proof. Pairs $(x, c)$ such that $x$ is
periodic of exact order~6 under the map $x \mapsto x^2 + c$ give rise
to points on the affine curve $Y^{\dyn}_1(6)$ with equation
\[ \Phi^*_6(x,c) := \frac{\bigl(f^{(6)}(x,c) - x\bigr)\bigl(x^2-x+c\bigr)}%
                         {\bigl(f^{(3)}(x,c)-x\bigr)\bigl(f^{(2)}(x,c)-x\bigr)}
                  = 0 \,,
\]
where $f^{(0)}(x,c) = x$, $f^{(n+1)}(x,c) = f^{(n)}(x^2+c,c)$ denote the
iterates of $x \mapsto x^2 + c$. (For some of the points, the orbit of~$x$
actually has exact order a proper divisor of~$6$; see~\cite{Silverman}
for details.) We denote by $X^{\dyn}_1(6)$ the smooth projective model
of~$Y^{\dyn}_1(6)$. This curve has an automorphism $\sigma$ of order~6
that is induced by the map $(x,c) \mapsto (x^2+c,c)$ on~$Y^{\dyn}_1(6)$.
We denote the quotient $X^{\dyn}_1(6)/\langle \sigma \rangle$ by
$X^{\dyn}_0(6)$. This is a curve of genus~4. We determine the set
of rational points on this curve, from which we can find the set of
rational points on~$X^{\dyn}_1(6)$. It turns out that all of these points
are ``cusps'', i.e., they are in the complement of~$Y^{\dyn}_1(6)$ and
hence do not correspond to pairs~$(x,c)$ as above.

We first find a nice model of $X^{\dyn}_0(6)$ (see Section~\ref{Smodel} below).
There are ten rational points on this curve that are easy to find.
We show that they generate a torsion-free subgroup $G$ of rank~3 in the
Mordell-Weil group~$J(\Q)$, where $J$ is the Jacobian of~$X^{\dyn}_0(6)$.
We further show that there are no other rational points that map into
the saturation of this subgroup in~$J(\Q)$. This is done in 
Section~\ref{Spoints} below. It remains to show that $G$ is a 
finite-index subgroup of~$J(\Q)$. It is in this part of the proof that
we have to make assumptions on the $L$-series, since we want to use the Birch
and Swinnerton-Dyer conjecture. We compute enough coefficients of
the $L$-series to show that its third derivative at $s = 1$ does not
vanish, which, according to the BSD conjecture, implies that the rank
of~$J(\Q)$ is at most~3. See Section~\ref{Srank} below. (Note that a
2-descent on~$J$, which is the usual way to obtain an upper bound on the
Mordell-Weil rank for low-genus curves, requires knowledge
of the class and unit groups of a number field of degree~119. The necessary
computations are utterly infeasible with current technology, even when
assuming GRH.)

We have used the {\sf MAGMA} \cite{Magma} computer algebra system in order
to perform the necessary computations. A script that can be loaded into
{\sf MAGMA} and that performs the relevant computations is available
at~\cite{Script}.

\medskip

This curve $X^{\dyn}_0(6)$ appears to be the first curve of genus~$\ge 4$ 
that is not very special in some way, e.g., hyperelliptic, or covering
a curve of smaller genus, or a modular or Shimura curve, for which the set
of rational points could be explicitly determined (assuming reasonable
standard conjectures). The methods used here should be applicable
in other cases as well, provided
\begin{itemize}
  \item we can find a finite-index subgroup of the Mordell-Weil group,
  \item its rank is less than the genus, and
  \item the conductor is reasonably small.
\end{itemize}

\subsection*{Acknowledgments}
I would like to thank the American Institute of Mathematics in Palo Alto,
California, for hosting a workshop on ``The uniform boundedness conjecture
in arithmetic dynamics'' in January~2008, and the organizers and participants
of that workshop for creating a very productive research environment.
Most of the computations described in this note were carried out during this
workshop.
I also would like to thank Fritz Grunewald for his suggestion to find the 
endomorphism ring of the Jacobian of~$X^{\dyn}_0(6)$.


\section{The Model} \label{Smodel}

In order to obtain a smooth projective model of $X^{\dyn}_0(6)$,
we first find an equation for $Y^{\dyn}_0(6)$ (the image of 
$Y^{\dyn}_1(6)$ in $X^{\dyn}_0(6)$) as an affine plane curve.
For a point $(x,c) \in Y^{\dyn}_1(6)$, we denote the ``trace'' of its orbit by
\[ t = x + (x^2 + c) + f^{(2)}(x,c) + \dots + f^{(5)}(x,c) \,. \]
The resultant with respect to~$x$
of $\Phi^*_6(x,c)$ and 
\hbox{$t - \bigl(f^{(0)}(x,c) + \dots + f^{(5)}(x,c)\bigr)$}
is a sixth power; one of its sixth roots is
\begin{align*}
  \Psi_6(t,c) &= 256(t^3 + t^2 - t - 1) c^3 
                 + 16(9 t^5 + 7 t^4 + 10 t^3 + 30 t^2 - 19 t - 37) c^2 \\
   &\quad{} + 8(3 t^7 + t^6 + 2 t^5 + 2 t^4 - 17 t^3 + 69 t^2 + 52 t - 48) c \\
   &\quad{} + t^9 - t^8 + 2 t^7 + 14 t^6 + 49 t^5 + 175 t^4 + 140 t^3 + 196 t^2 
           + 448 t \,.
\end{align*}
We first resolve the singularities at infinity. Successively getting
rid of multiple factors at the edges of the Newton polygon, we arrive at
the equation
\begin{align*}
  F(u,v) &= (u^4 - u^3) v^3 + (-u^5 + 9 u^4 + 6 u^3 - 17 u^2 + 3 u) v^2 \\
   &\qquad{} + (4 u^4 + 74 u^3 - 52 u^2 - 54 u + 24) v \\
   &\qquad{} + 4 u^4 + 24 u^3 + 117 u^2 - 261 u + 72 \\
         &= 0 \,.
\end{align*}
Here
\[ u = \frac{2}{t+1} \quad\text{and}\quad
   v = 4(c-1) + (t-1)^2 + \frac{2}{t+1} \,,
\]
or
\[ t = \frac{2}{u} - 1 \quad\text{and}\quad
   c = \frac{v}{4} - \frac{1}{u^2} + \frac{2}{u} - \frac{u}{4} \,.
\]
The curve defined by this equation has three singularities at points
$(\alpha, \beta)$, where
\[ 3 \beta^3 + 32 \beta^2 + 69 \beta + 72 = 0 \quad\text{and}\quad
   18 \alpha = 6 \beta^2 + 55 \beta + 69 \,.
\]
From the Newton polygon of~$F$, we see that the regular differentials
on the smooth projective model of this curve are contained in the
space spanned by 
$\{\omega_0, u \,\omega_0, u v \,\omega_0, u^2 \,\omega_0, u^2 v \,\omega_0,
   u^3 \,\omega_0, u^3 v \,\omega_0\}$,
where 
\[ \omega_0 = \frac{du}{\frac{\partial}{\partial v} F(u,v)}
            = -\frac{dv}{\frac{\partial}{\partial u} F(u,v)} \,.
\]
(See~\cite{Khovanskii}, in particular the example on page~42.)
These differentials are regular everywhere except perhaps at the
singularities described above. In order to get something regular there,
the polynomial that $\omega_0$ is multiplied by has to vanish at the
singularities. We obtain the following basis of~$\Omega^1_{X^{\dyn}_0(6)}$.
\begin{align*}
  \omega_1 &= (u^3 v + 2 u^3 - 3 u^2 v - u^2 + 3 u v + 6) \,\omega_0 \\
  \omega_2 &= (u^3 v + 2 u^3 - u^2 v + u^2 + 3 u - 6) \,\omega_0 \\
  \omega_3 &= (u^3 v + 2 u^3 - 4 u^2 v - 3 u^2 + 3 u v - 3 u) \,\omega_0 \\
  \omega_4 &= (3 u^3 v + 4 u^3 - 3 u^2 v + 6 u^2 - 6 u) \,\omega_0
\end{align*}
The canonical model of a curve of genus~4 is the intersection of a
quadric and a cubic in~$\BP^3$. We see that $u$ is a rational function
of degree~3, which implies that the quadric splits, i.e., it is
isomorphic to $\BP^1 \times \BP^1$ over~$\Q$. So there is a model 
of~$X^{\dyn}_0(6)$ that is a smooth curve of bidegree~$(3,3)$ 
in~$\BP^1 \times \BP^1$. To find a suitable second coordinate (besides~$u$),
we take the quotient of two differentials vanishing on $u = 0$. This
means that the differentials may not contain $\omega_0$ or $u v\,\omega_0$
with a nonzero coefficient. A possible choice is
\[ w = \frac{-\omega_1 - \omega_2 + \omega_3 + 2 \omega_4}{\omega_4}
     = \frac{-u^2 v + 3 u v + 18}{u^2 v + 2 u^2 + 3 u + 6} \,. 
\]
In terms of $u$ and~$w$, we now have $X^{\dyn}_0(6)$ as a smooth curve
in~$\BP^1 \times \BP^1$, with (affine) equation
\[ G(u,w) 
    = w^2 (w+1) u^3 - (5w^2 + w + 1) u^2 - w(w^2 - 2 w - 7) u
       + (w + 1)(w - 3) = 0 \,.
\]
We will denote this curve by~$C$.
Note that
\[ c = \frac{(-u^3 - 2 u^2 + 5 u - 10) u w 
               - u^4 + 3 u^3 + 8 u^2 - 10 u + 12}%
            {4 u^2 (u w + u - 3)}
\]
on this model. 

Our model has good reduction except at~$2$ and at~$p = 8029187$. 
Mod~$p$, we have a node at $(u,w) = (2937959, 7887180)$ with tangent
directions defined over~$\F_p$. This point is regular on
the arithmetic surface given by our equation.

Mod~2, there is a node at $(1,0)$
with tangent directions defined over~$\F_4$ and a tacnode at $(0,1)$
with local branches again defined over~$\F_4$. Both singularities are
non-regular points of the arithmetic surface.
Resolving these points gives us the minimal proper regular model over~$\Z_2$.
The node resolves into a chain of three $\BP^1$'s whose ends intersect
the original component. Blowing up the tacnode gives a double line, all
of whose points are non-regular. Blowing up this line, we obtain a smooth
curve of genus~1, meeting the original components in two (regular) points.
Therefore, the special fiber of the minimal proper regular model 
consists of five components $A,B,C,C',D$, each of
multiplicity one. $A$ and $B$
are both elliptic curves with trace of Frobenius~$-1$, the other
components are $\BP^1$'s. $A$, $B$, and~$D$ are defined over~$\F_2$,
$C$ and~$C'$ are defined over~$\F_4$ and conjugate. The intersection
matrix is as follows.
\[ \begin{array}{r|ccccc|}
        & A & B & C & C' & D \\\hline
     A  & -4 & 2 & 1 & 1 & 0 \\
     B  & 2 & -2 & 0 & 0 & 0 \\
     C  & 1 & 0 & -2 & 0 & 1 \\
     C' & 1 & 0 & 0 & -2 & 1 \\
     D  & 0 & 0 & 1 & 1 & -2 \\\hline
   \end{array}
\]
(For some worked examples of how to compute minimal regular models,
see for example~\cite{FLSSSW} or~\cite{PSS}.)

The two (separate) intersection points of $A$ and~$B$ are swapped by
the action of~Frobenius. We see that the reduction of the Jacobian
has a 2-dimensional abelian and a 2-dimensional toric component
(since the dual graph of the special fiber has two independent loops,
compare~\cite[\S~9.2]{BLR});
Frobenius reverses the orientation of both loops.
We can summarize our findings in the following lemma.

\begin{Lemma}
  The Jacobian of $C = X^{\dyn}_0(6)$ has conductor $2^2\,p$, where
  $p = 8\,029\,187$ is the big prime from above. Its Euler factor at~$2$
  is $(1 + T + 2T)^2 (1 + T)^2$.
\end{Lemma}

We end this section by showing that $X_0^{\dyn}(6)$ does not have any
special geometrical properties that might help us.

\begin{Lemma} \label{L2} \strut
  We have $\operatorname{End}_{\bar{\Q}} J = \Z$. In particular:  
  \begin{enumerate}
    \item The automorphism group of $X_0^{\dyn}(6)$ is trivial
          (even over~$\bar{\Q}$).
    \item The Jacobian of $X_0^{\dyn}(6)$ is absolutely simple.
    \item The curve $X_0^{\dyn}(6)$ does not cover any other curve of positive
          genus, except itself (not even over~$\bar{\Q}$).
  \end{enumerate}
\end{Lemma}

\begin{proof}
We take inspiration from the proof of Prop.~9 in~\cite{FPS}. To make
matters more concrete, we formulate a computational lemma.

\begin{Lemma} \label{L3}
  Let $C$ be a curve of genus~$g$ over a number field~$K$, with Jacobian~$J$, 
  let $v$ be a finite place of~$K$ of good reduction for~$C$, and let
  $f(T)$ be the Euler factor (i.e., the numerator of the zeta function)
  of~$C$ at~$v$. If $f \in \Q[T]$ is irreducible, and no monic irreducible 
  factor of
  \[ h(T) = \frac{\operatorname{Res}_x\bigl(f(x), f(Tx)\bigr)}{(1-T)^{2g}} \]
  has integral coefficients and constant term~$1$, 
  then $\operatorname{End}_{\bar{K}} J$
  embeds into the number field generated by a root of~$f$.
\end{Lemma}

\begin{proof}
For the proof, note that the roots of~$h$ are all the quotients $\alpha/\beta$,
where $\alpha$ and~$\beta$ are distinct roots of~$f$. If one of these quotients
is a root of unity, then $h$ has a monic irreducible factor that 
has integral coefficients and constant term~$1$ 
(namely, some cyclotomic polynomial).
Conversely, if there is such an irreducible factor, then its roots are
units in the splitting field of~$f$, and they have absolute value~$1$ in
all complex embeddings (since $|\sigma(\alpha)| = q^{-1/2}$ for all
complex embeddings~$\sigma$ and all roots~$\alpha$ of~$f$, where $q$ is the
size of the residue class field~$k_v$). Hence some $\alpha/\beta$ is a root of
unity. 

Now this is the case if and only if, for some $n \ge 2$, there are distinct
roots $\alpha$ and~$\beta$ of~$f$ such that $\alpha^n = \beta^n$. This in
turn is equivalent to the Galois orbit of~$\alpha^n$ having size less
than~$\deg f = 2g$, which means that the characteristic polynomial of the 
$n$th power of the $v$-Frobenius 
is not irreducible. 

Our assumptions therefore imply that all these characteristic
polynomials are irreducible. (An argument like this was used
in~\cite{StollSimple} to show that certain genus~2 Jacobians are absolutely
simple.) From \cite[Thm.~8]{WaterhouseMilne}, we
then see that the endomorphism algebra of~$J$ over~$\bar{k}_v$
is the number field generated by a root of~$f$, and since the endomorphism
ring of~$J$ (over~$\bar{K}$) embeds into this algebra, the claim is proved.
(Note that $J$ is simple over~$k_v$ since $f$ is irreducible.)
\end{proof}

To prove Lemma~\ref{L2}, we compute the Euler factors at $p = 5$ and~$p = 7$.
They are
\[ 1 + 3 T + 6 T^2 + 6 T^3 - 8 T^4 + 30 T^5 + 150 T^6 + 375 T^7 + 625 T^8 \,. \]
and
\[ 1 + 7 T + 28 T^2 + 94 T^3 + 276 T^4 + 658 T^5 + 1372 T^6 + 2401 T^7 
   + 2401 T^8 \,.
\]
We observe that both polynomials satisfy the assumptions in Lemma~\ref{L3}
and that the number fields they generate are linearly disjoint over~$\Q$.
This proves the first claim.

Statement~(1) then follows, since any nontrivial automorphism of the curve 
would induce a nontrivial automorphism of the Jacobian~$J$. But the only 
nontrivial automorphism of~$J$ is multiplication by~$-1$, and if it would
come form an automorphism of the curve, this would imply that the curve
is hyperelliptic, which is not the case. Alternatively, we can use the
fact that any automorphism of our curve
must extend to an automorphism of~$\BP^1 \times \BP^1$. Such automorphisms
either perform a M\"obius transformation on each of the factors separately,
or else this type of automorphism is followed by swapping the two factors.
A Gr\"obner basis computation shows that the only automorphism 
of~$\BP^1 \times \BP^1$ that fixes the curve is the identity.

If statement~(2) were false, then the algebra 
$\Q \cong \Q \otimes_{\Z} \operatorname{End}_{\bar{\Q}} J$ would have zero 
divisors, which is not the case.

Finally, if the curve covers another curve of positive genus and the
map is not an isomorphism, then the other curve has genus strictly
between 0 and~4. But then its Jacobian will be a factor of the Jacobian
of~$X_0^{\dyn}(6)$, so the latter would have to split, contradicting
the fact that~$J$ is absolutely simple.
\end{proof}


\section{Rational Points} \label{Spoints}

A quick search finds the following ten rational points on~$C$.

\[ \begin{array}{|c|cc|cc|} \hline
      & u & w & t & c \\\hline
     P_0 \text{\large\strut}
         & 0 & \infty & \infty & \infty \\
     P_1 \text{\large\strut}
         & 0 & -1 & \infty & \infty \\
     P_2 \text{\large\strut}
         & 0 & 3 & \infty & \infty \\
     P_3 \text{\large\strut}
         & \infty & 0 & -1 & \infty \\
     P_4 \text{\large\strut}
         & 1 & 2 & 1 & \infty \\\hline
   \end{array}
   \qquad
   \begin{array}{|c|cc|cc|} \hline
      & u & w & t & c \\\hline
     P_5 \text{\large\strut}
         & 2 & 1 & 0 & 0 \\
     P_6 \text{\large\strut}
         & 1 & \infty & 1 & -2 \\
     P_7 \text{\large\strut}
         & \infty & -1 & -1 & -2 \\
     P_8 \text{\large\strut}
         & -1 & \infty & -3 & -4 \\
     P_9 \text{\large\strut}
         & -\tfrac{4}{5} & -1 & -\tfrac{7}{2} & -\tfrac{71}{48} \\\hline
   \end{array}
\]

The first five of these are the ``cusps''; these are the points that
have to be added to $Y_0^{\dyn}(6)$ in order to obtain a smooth projective
curve. It is known that all cusps on~$X_1^{\dyn}(N)$ and hence also 
on~$X_0^{\dyn}(N)$ are rational points, for all~$N$.

The remaining five points correspond to cycles of length~6 for the
given value of~$c$ that are stable as a set (or as a cycle) under the
action of the absolute Galois group of~$\Q$. For the special values
$c = 0$ and $c = -2$, these cycles are ``predictable''; they come
from roots of unity. For $N = 6$, we find cycles containing
$\zeta_9$ when $c = 0$ and cycles containing $\zeta_{13} + \zeta_{13}^{-1}$
(this is the one whose trace~$t$ is~$-1$)
or $\zeta_{21} + \zeta_{21}^{-1}$ (with $t = 1$) when $c = -2$. 
(We use $\zeta_n$ to
denote a primitive $n$th root of unity; these cycles then contain all
possible values of the above expressions.)

For $c = -4$, the points in the cycle live in a sextic abelian number field
with discriminant $5^3 \cdot 7^4$ and conductor~$35$; it is the field
$\Q(\sqrt{5}, \cos\tfrac{2\pi}{7})$. Finally, for $c = -\tfrac{71}{48}$,
we find points defined over the {\em quadratic} field $\Q(\sqrt{33})$;
one point in the cycle is $x = -1 + \tfrac{1}{12} \sqrt{33}$. In particular,
this means that $X_1^{\dyn}(6)\bigl(\Q(\sqrt{33})\bigr)$ contains an orbit
of six non-cuspidal points.

See also the end of~\cite{FPS}.

\medskip

We will now prove the following result.

\begin{Lemma}
  Let $J$ denote the Jacobian of~$C = X_0^{\dyn}(6)$.
  \begin{enumerate}
    \item $J(\Q)$ has trivial torsion subgroup.
    \item The subgroup $G$ of $J(\Q)$ generated by the classes of divisors
          supported in the 10 rational points listed above is isomorphic
          to~$\Z^3$.
    \item The subgroup is already generated by divisors supported at
          the cusps.
  \end{enumerate}
\end{Lemma}

\begin{proof}
We know that 
the prime-to-$p$ torsion in $J(\Q)$ injects into $J(\F_p)$ for primes
of good reduction, so the observation that (as computed by {\sf MAGMA})
\[ \#J(\F_7) = 2 \cdot 7 \cdot 11 \cdot 47 \text{\quad and\quad}
   \#J(\F_{13}) = 3 \cdot 17 \cdot 23 \cdot 43
\]
shows that $J(\Q)$ has trivial torsion subgroup.

The main tool for proving the other assertions is the homomorphism
\[ \Phi_S : \bigoplus_{i=0}^9 \Z P_i \To \Pic_C 
            \To \prod_{p \in S} \Pic_{C/\F_p} \,,
\]
where $S$ is a set of primes of good reduction. 
We take $S = \{3, 5, 7, 11, 13\}$ and compute the kernel of~$\Phi_S$.
This kernel is a subgroup of rank~9 in~$\Z^{10} = \bigoplus \Z P_i$.
We apply LLL to it and find that there are six independent elements
with very small coefficients (and three large additional basis vectors).
We suspect that the small elements come from actual relations
between our points; this can then be verified by exhibiting a suitable
rational function. ({\sf MAGMA} provides the necessary functionality
for these computations.) Denoting linear equivalence by `$\sim$',
we find the following six independent relations.
\begin{align*}
  P_0 + P_6 + P_8 &\sim P_1 + P_7 + P_9 \\
  P_0 + P_1 + P_2 &\sim 2 P_3 + P_7 \\
                  &\sim 2 P_4 + P_6 \\
  P_0 + P_2 + P_7 + P_9 &\sim P_1 + P_3 + P_5 + P_6 \\
  2 P_0 + P_1 + P_6 &\sim P_2 + P_3 + 2 P_5 \\
  3 P_0 + P_3 &\sim P_1 + P_2 + P_6 + P_8
\end{align*}

On the other hand, looking at the image of~$\Phi_S$, we see that the
degree~0 subgroup of~$\Z^{10}$ surjects onto $(\Z/3\Z)^3$. Since we
know that there is no torsion in~$J(\Q)$, this implies that the rank
of the image of the degree~0 subgroup in~$J(\Q)$ must be at least~3.
The existence of the relations above implies that the rank is at most~3,
so the rank is exactly~3, and since there is no torsion, the group
must be isomorphic to~$\Z^3$.

Finally, from the relations we have given it is easy to verify that
$P_3$ and $P_5, \dots, P_9$ can be expressed in terms of $P_0$, $P_1$, $P_2$,
and~$P_4$. This means that our subgroup is already generated by divisors
supported at the latter four points (all of which are cusps). The only
relation between the cusps is
\[ 5 P_0 - 10 P_1 - 2 P_2 + P_3 + 6 P_4 \sim 0 \,; \]
it is perhaps worth noting that this relation is {\em not} induced by
the standard ``dynamical units'' as provided by~\cite[Thm.~2.33]{Silverman}
or~\cite{Narkiewicz}, applied to the coordinate ring of~$Y_1^{\dyn}(6)$.
\end{proof}

Our next result is as follows.

\begin{Lemma}
  The ten points $P_0, \dots, P_9$ are the only rational points
  whose images in $\Pic_C$ are in the saturation of the
  subgroup described in the previous lemma.
\end{Lemma}

\begin{proof}
We use Chabauty's method (see for example 
\cite{Chabauty,Coleman,McCallumPoonen,StollChab})
for the proof. Recall that there is a pairing
\[ \Omega^1_J(\Q_p) \times J(\Q_p) \To \Q_p\,, \qquad
   (\omega, Q) \longmapsto \int_0^Q \omega
\]
that induces a perfect $\Q_p$-bilinear pairing
\[ \Omega^1_J(\Q_p) \times J(\Q_p)^1 \otimes^{\mathstrut}_{\Z_p} \Q_p \To \Q_p \,, \]
where $J(\Q_p)^1$ denotes the kernel of reduction. If $G \subset J(\Q_p)$
is a subgroup of rank less than $\dim J = 4$, then there must be a
nonzero differential $\omega$ that kills~$G$ under this pairing.
Note that $\omega$ then also kills the saturation
\[ \bar{G} = \{P \in J(\Q_p) : nP \in G \text{\ for some $n \ge 1$}\} \]
of~$G$.
We will apply this with $p = 5$ and $G$ the subgroup generated by the
known rational points as above.

For points in the
kernel of reduction, the integral can be evaluated by formally integrating
the power series representing~$\omega$ in terms of a system of local
parameters at the origin and then plugging in the values at~$Q$ of these
parameters. For practical computations, it is more convenient to use
the canonical identification $\Omega^1_J(\Q_p) \cong \Omega^1_C(\Q_p)$.

Let $P' \in C(\Q)$ be a fixed base-point. Let $Q \in J(\Q_p)^1$. Then
$Q$ is represented by a divisor of the form
$(Q_1+Q_2+Q_3+Q_4) - 4 P'$, where the points $Q_j \in C(\bar{\Q}_p)$ all 
reduce to~$P'$ modulo the prime above~$p$ in their field of definition.
Let $\tau$ be a uniformizer at~$P'$ that reduces mod~$p$ to a uniformizer
at the reduction of~$P'$. The differential
$\omega$ can be written as $\phi(\tau)\,d\tau$
with a power series~$\phi \in \Q_p[\![T]\!]$. Let 
\[ \lambda = \lambda_1 T + \lambda_2 T^2 + \dots \] 
be its formal integral. Then
\[ \int_0^Q \omega = \sum_{j=1}^4 \lambda\bigl(\tau(Q_j)\bigr)
                   = \sum_{n=1}^\infty \lambda_n \sum_{j=1}^4 \tau(Q_j)^n \,;
\]
the series converges in~$\Q_p$. Note that the power sums can be computed
from the coefficients of the characteristic polynomial
\[ \bigl(X - \tau(Q_1)\bigr)\bigl(X - \tau(Q_2)\bigr)
   \bigl(X - \tau(Q_3)\bigr)\bigl(X - \tau(Q_4)\bigr) \,,
\]
which lie in the field of definition of~$Q$.

In our concrete case, we take $P' = P_1$. Applying LLL to the kernel
of the reduction map $\bigoplus_{j \neq 1} \Z(P_j - P_1) \to J(\F_5)$,
we find a basis of $G \cap J(\Q_5)^1$, given by
\begin{align*}
  D_1 &= P_7 - P_9 \\
  D_2 &= P_0 - 6 P_1 + 2 P_5 + P_7 + P_8 + P_9 \\
  D_3 &= P_0 - 3 P_1 + 2 P_2 + P_4 + P_6 - P_7 - P_8
\end{align*}
For each of these, we find $D'_j$ such that $D_j \sim D'_j - 4 P'$ and
$D'_j$ is effective of degree~4, with points reducing to~$P'$.
The point $P' = P_1$ has coordinates $(u,w) = (0,-1)$; we can choose
$u$ as a uniformizer at~$P'$ and its reduction. The space of regular 
differentials is spanned by
\[ \omega_0 = \frac{du}{\frac{\partial}{\partial w} G(u,w)}\,, \quad
   \omega_1 = u\, \omega_0\,, \quad
   \omega_2 = w\, \omega_0\,, \quad\text{and}\quad
   \omega_3 = u w\, \omega_0 \,.
\]
We expand each $\omega_i$ as a power series in~$u$ times~$du$ and
let $\lambda_i \in u\Q[\![u]\!]$ be its formal integral. Then we
evaluate each $\lambda_i$ at each $D'_j$ as described above. We determine
the kernel of the resulting matrix, which gives us the differential~$\omega$
that kills our subgroup~$G$. We find that reduced mod~5, this differential
is $\bar{\omega} = \bar{\omega}_2$. It vanishes at the points where
$w = 0$ or $u = \infty$. There are two such points in $C(\F_5)$, namely
$(\infty, -1)$ and $(\infty, 0)$. At the former, $\bar{\omega}$ vanishes
to first order, which implies that there are at most two rational points
in that residue class (see for example~\cite{StollChab}). 
Since we have the points $P_7 = (\infty, -1)$
and $P_9 = (-4/5, -1)$, these must be all the rational points in this
residue class. At $(\infty, 0)$, we compute explicitly that the
logarithm~$\lambda$ that vanishes on~$C(\Q)$ on this residue class is
\[ \lambda = \gamma \tau \bigl(1 - (2 + O(5)) 5 \tau + O(5^2)\bigr) \]
with some constant~$\gamma \neq 0$, where $5\tau$ is the 
uniformizer~$w$ at $(\infty, 0)$. So $\lambda$ has a single zero on
this residue class, which is taken care of by $P_3 = (\infty, 0)$.
On all other points in~$C(\F_5)$, $\bar{\omega}$ does not vanish, hence
there can be at most one rational point in each of these residue classes.
Since it is easily checked that $\{P_0, \dots, P_9\} \to C(\F_5)$ is 
surjective, this shows that there are no other points $P$ in~$C(\Q)$
such that $P - P'$ is in~$G$. In fact, there is no such point that maps
into the saturation of $G$ in~$J(\Q_5)$ (since $\omega$ kills~$\bar{G}$). 
So there is no rational point on~$C$ mapping
into the saturation of~$G$ other than those already known.
\end{proof}

\begin{Theorem}
  If $\rank J(\Q) = 3$, then $X_0^{\dyn}(6)$ has only the ten rational
  points listed above. In particular, it then follows that the only
  rational points on~$X_1^{\dyn}(6)$ are the cusps, so that there is
  no cycle of exact length~6 consisting of rational numbers under
  an iteration $x \mapsto x^2 + c$.
\end{Theorem}

\begin{proof}
  If $J(\Q)$ has rank~3, then $J(\Q)$ is the saturation
  of~$G$, the subgroup generated by degree 0 divisors supported
  on the known rational points, since the latter then has finite index
  in~$J(\Q)$. The previous lemma then shows that there are no other
  rational points on~$X_0^{\dyn}(6)$ than those already known. None
  of the non-cuspidal points among these lift to a rational point
  on~$X_1^{\dyn}(6)$, so the latter curve can have no non-cuspidal 
  rational points. A rational 6-cycle would give rise to a non-cuspidal
  rational point on this curve, so such a rational 6-cycle cannot exist.
\end{proof}


\section{Bounding the Rank} \label{Srank}

It remains to show that $\rank J(\Q) = 3$. We know that the rank is
at least~3, so it suffices to show that it is at most~3.

The standard procedure for obtaining an upper bound for the rank is
a descent on the Jacobian. However, the complexity of this quickly
becomes prohibitive when the genus is not very small and the curve
does not have any helpful special features. For example, 2-descent on Jacobians
of general non-hyperelliptic genus~3 curves is still in its infancy and
so far has been successful in only one example (assuming GRH for the
computation). Here, we have a curve of genus~4, and it appears that there
are no helpful special properties, see Lemma~\ref{L2} above. 
Usually, our best bet is a 2-descent,
and for this, the best approach seems to be to look at the odd theta
characteristics (whose differences generate the 2-torsion subgroup).
On a curve of genus~4, there are 120 of them; they correspond to
$(1,1)$-forms on~$\BP^1 \times \BP^1$ that meet the curve tangentially
in three points (more precisely, the intersection divisor is twice
an effective divisor of degree~3). We can set up the scheme describing
these; after a Gr\"obner basis computation, we find that it has one
rational point, and the other 119~points form a single Galois orbit.
This means that in order to do anything in the direction of a 2-descent,
we would have to compute the ideal class group and fundamental units
of a number field of degree~119. Before we are able to perform such
computations (even if we allow ourselves to assume GRH), we need very
substantial progress in the development of suitable algorithms.

The result on the Galois orbits on the odd theta characteristics can 
be obtained faster by computing the Gr\"obner bases
of the scheme over~$\F_5$ and over~$\F_{13}$. We first note that
$1 + u + w$ gives rise to the point defined over~$\Q$ (the $(1,1)$-forms
are of the form $a + b u + c w + d uw$). Over~$\F_5$, the remaining 119~points
split into nine orbits of length~7 and four orbits of length~14,
whereas over~$\F_{13}$, they split into seven orbits of length~17.
Since these partitions must refine the orbit partition over~$\Q$,
there must be a single orbit of length~119.

We can extract some more information. First note that the theta
characteristics can be identified with the 2-torsion subgroup~$J[2]$ 
(by sending the unique
odd theta characteristic that is defined over~$\Q$ to the origin). The Galois
action on~$J[2]$ must then have orbits of lengths~1 and~119. We will
determine the image of Galois in~$\operatorname{Sp}_8(\F_2)$ and deduce
that the remaining 136~elements also form a single orbit.

The Frobenius automorphisms at $p = 5$ and~$13$, acting on~$J[2]$, have
orders $14$ and~$17$, as we saw above. There is only one maximal
subgroup~$\Gamma$ of~$\operatorname{Sp}_8(\F_2)$ (up to conjugacy) 
whose order is a multiple of $14 \cdot 17$. There is a maximal subgroup
of~$\Gamma$ with this property, but it does not contain elements
of order~$14$. Since the
action of the full group~$\operatorname{Sp}_8(\F_2)$ is transitive
on~$J[2] \setminus \{0\}$, the image of the Galois group
in~$\operatorname{Sp}_8(\F_2)$ must be~$\Gamma$ (which is a subgroup
of index~120). 
Since the action of~$\Gamma$ on~$J[2]$ has orbits
of lengths 1, 119, and~136, our claim follows. It can be checked that
the smallest faithful permutation representation of~$\Gamma$ has
degree~119, so that this is really the smallest possible degree of
a number field that we can hope for in a 2-descent computation.

Note also that we showed in Lemma~\ref{L2} that $J$ has no
endomorphisms other than the multiplication-by-$n$ maps, so that
multiplication-by-2 is the isogeny $J \to J$ of lowest possible degree that
can be used for a descent argument. There are no nontrivial 
Galois-stable subgroups of~$J[2]$, so there are no 2-isogenies to
other abelian varieties either.

\smallskip

A possible alternative approach to obtaining a bound for the rank
assumes the (weak) Birch and Swinnerton-Dyer
conjecture (plus standard conjectures on analytic continuation and
functional equations of $L$-series, see for example 
\cite[Conjs. 2.8.1. and~3.1.1]{Huls}). The conjecture predicts that
the rank of~$J(\Q)$ is the same as the order of vanishing of the
$L$-series $L(J, s) = L(C, s)$ at $s = 1$. In order to be able
to evaluate the $L$-series and its derivatives there, we need to
compute its coefficients $a_n$ for values of~$n$ up to a suitable
multiple of the square root of its conductor. Luckily, in our case
the conductor $2^2 \cdot 8\,029\,187$ is not too large, so that we
can actually perform the computation in reasonable time.

We use Tim Dokchitser's $L$-series package~\cite{Dokchitser} in its {\sf MAGMA}
implementation. We will not need to find the Euler factor at the
large bad prime (it is beyond the necessary range of coefficients).
For the other bad prime~2, we found the Euler factor in Section~\ref{Smodel}.
For the good primes, we need the Euler factor up to $T^d$, where
$d = \lfloor \log_p N \rfloor$ and $N$ is the number of coefficients
required. This information can be obtained by counting the number
of points in $C(\F_{p^e})$ for $e = 1, \dots, \min\{d, 4\}$. For a precision
of~$10^{-20}$, we need 183997 coefficients (which we can compute in a
day or so). We verify numerically that our $L$-series 
satisfies the functional
equation it is supposed to satisfy (with sign~$-1$). Then we find
that the $L$-series and its first two derivatives vanish at~$s = 1$
to the given precision, whereas $L'''(C, 1) = 0.83601\dots$ is clearly
nonzero. Assuming the weak Birch and Swinnerton-Dyer conjecture for~$J$,
this implies that $\rank J(\Q) \le 3$. We therefore obtain our main
result below.

\begin{Theorem}
  Let $J$ be the Jacobian of~$X_0^{\dyn}(6)$. If the $L$-series
  $L(J,s)$ extends to an entire function and satisfies the
  standard functional equation, and if the weak Birch and Swinnerton-Dyer
  conjecture is valid for~$J$, then there are no rational cycles of
  exact length~6 under $x \mapsto x^2 + c$.
\end{Theorem}


\section{What Next?}

Without fundamentally new ideas, it seems unlikely that we can make
our result unconditional in the foreseeable future. In another direction,
it looks rather hopeless to try to get a similar result for $X_0^{\dyn}(7)$.
This curve has genus~16 and bad reduction at the 35-digit prime
$p = 84562\,62122\,13597\,75358\,18884\,16725\,49561$ and possibly at~2.
In any case, the conductor will be very large (at least~$p$) and so 
there will be no chance whatsoever to use the $L$-series numerically
to obtain information
on the rank. It might still be possible to get some information on the
subgroup of the Jacobian generated by the cusps (e.g., by making use of 
dynamical units). It will be very hard, however, to use this information
for a Chabauty argument, for example.

Another question is how the large bad primes can be explained or even
predicted. We have $3701$ for $N = 5$ (the only bad prime for~$X_0^{\dyn}(5)$,
see~\cite{FPS}; their model is also bad at~2, but this can easily be
repaired), $8029187$ for $N = 6$ and the
35-digit prime above for $N = 7$. Note that unless we can shed light
on this question, it is likely to be very hard to try and prove algebraically
that $Y_1^{\dyn}(N)$ is smooth, since such a proof must break down
when the characteristic is one of these primes.

\medskip

The following could be a possible line of attack for a proof that
$Y_1^{\dyn}(N)(\Q)$ is empty for large~$N$. There is a good description
of the formal neighborhoods of the cusps on~$X_1^{\dyn}(N)$, using
symbolic dynamics. If we could use this to prove, for any odd prime~$p$,
that the cusp is the only rational point in its residue class mod~$p$,
and also to prove a similar statement modulo a suitable power of~2,
then this would imply that a parameter $c \in \Q$ that allows for a
cycle of exact length~$N$ of rational numbers must be essentially integral
(more precisely, its denominator must divide a fixed power of~2). It is
then fairly easy to show that $N$ is bounded. 

For example, assume that
$c \in \Z$ and $x$ is in a cycle. Then $x$ must also be an integer.
If $c > 0$, we have $x^2 + c > |x|$, so there cannot be a cycle. For
$c = 0$, the only possibilities are $x = 0$ and~$x = 1$. For $c = -1$, 
the only possibilities
are $x = 0$ and~$x = -1$. If $c < -1$, we have $x^2 + c > |x|$
whenever $|x| \ge \sqrt{|c|} + 1$. So we must have $|x| < \sqrt{|c|} + 1$.
But then we also need that $|x^2 + c| < \sqrt{|c|} + 1$, which implies
that $\sqrt{|c|-\sqrt{|c|}-1} < |x| < \sqrt{|c|}+1$. This interval has
length less than~2, so there are at most two possible values for~$|x|$.
This implies that there must be either a fixed point or a cycle of length~2.
Indeed, for any $x \in \Z$, $x$ is a fixed point for $c = x - x^2$, and
$(x, -x-1)$ is a 2-cycle for $c = -x^2-x-1$.

In the more general case when $m^2 c \in \Z$ for some fixed integer $m \ge 1$,
we can use similar arguments to show that the cycle must be contained
in a union of a bounded number of intervals whose lengths are bounded.
Since the possible values of~$x$ are in the set $\frac{1}{m}\Z$,
there must be a bound on their number.

In the spirit of the methods used in this paper, the necessary result for odd
primes~$p$ would follow from the following two statements.
\begin{enumerate}
  \item The cuspidal group (i.e., the group generated by degree zero
        divisors supported at the cusps) is of finite index in the
        Mordell-Weil group of the Jacobian of~$X_1^{\dyn}(N)$.
  \item For every cusp~$P\!$, there is a regular differential on~$X_1^{\dyn}(N)$,
        defined over~$\Q_p$, that kills the cuspidal group and whose
        reduction mod~$p$ does not vanish at~$P$.
\end{enumerate}
However, we need some additional ideas to approach a proof of either one
of these.


\end{document}